\documentclass{compositio}
\usepackage{amsmath,amssymb,enumerate}
\newcommand{\FF}{{\mathbb{F}}}
\newcommand{\CC}{{\mathbb{C}}}
\newcommand{\QQ}{{\mathbb{Q}}}
\newcommand{\ZZ}{{\mathbb{Z}}}

\newcommand{\bC}{{\mathbf{C}}}
\newcommand{\bG}{{\mathbf{G}}}
\newcommand{\bT}{{\mathbf{T}}}

\newcommand{\Ind}{{\operatorname{Ind}}}
\newcommand{\St}{{\operatorname{St}}}
\newcommand{\reg}{{\operatorname{reg}}}
\newcommand{\Irr}{{\operatorname{Irr}}}
\newcommand{\Lie}{{\operatorname{Lie}}}

\newcommand{\PSL}{{\operatorname{L}}}
\newcommand{\PSU}{{\operatorname{U}}}
\newcommand{\SL}{{\operatorname{SL}}}
\newcommand{\Sp}{{\operatorname{Sp}}}
\newcommand{\PSp}{{\operatorname{S}}}

\newcommand{\SO}{{\operatorname{SO}}}
\newcommand{\Om}{{\operatorname{\Omega}}}
\newcommand{\Spin}{{\operatorname{Spin}}}
\newcommand{\Gss}{{G_\mathrm{ss}^*}}

\newcommand{\Chevie}{{\sf Chevie}}

\newcommand{\tw}[1]{{}^#1\!}

\let\eps=\epsilon
\let\vhi=\varphi

\newtheorem{thm}{Theorem}[section]
\newtheorem{lem}[thm]{Lemma}

\newtheorem{cor}[thm]{Corollary}

\newtheorem{prop}[thm]{Proposition}

\theoremstyle{definition}

\theoremstyle{remark}
\newtheorem{rem}[thm]{Remark}
\newtheorem{rems}[thm]{Remarks}

\begin{document}

\title[Rigidity]
{Rational Rigidity for $E_8(p)$}

\date{\today}

\author{Robert Guralnick}
\email{guralnic@usc.edu}
\address{Department of Mathematics, University of
  Southern California, Los Angeles, CA 90089-2532, USA}
\author{Gunter Malle}
  \email{malle@mathematik.uni-kl.de}
\address{FB Mathematik, TU Kaiserslautern, Postfach 3049,
  67653 Kaisers\-lautern, Germany}

\thanks{The first author was partially supported by the NSF
grant DMS-1001962  and the  Simons Foundation
Fellowship 224965.  The second author gratefully acknowledges financial support
by ERC Advanced Grant 291512.}

\begin{abstract}
We prove the existence of certain rationally rigid triples in
$E_8(p)$ for good primes $p$ (i.e. $p > 5$)
 thereby showing that these groups occur as
Galois groups over the field of rational numbers. We show that these triples
give rise to rigid triples in the algebraic group and prove that they
generate an interesting subgroup in characteristic $0$. As a byproduct
of the proof, we
derive a remarkable symmetry between the character table of a finite reductive
group and that of its dual group.    We also give a short list of possible overgroups
of regular unipotent elements in exceptional groups.
\end{abstract}

\maketitle


\section{Introduction} \label{sec:intro}

The question on which finite groups occur as Galois groups over the field
of rational numbers is still wide open. Even if one restricts to the case
of finite non-abelian simple groups, only rather few types have been realized
as Galois groups over $\QQ$. These include the alternating groups, the
sporadic groups apart from $M_{23}$, and some families of groups of Lie type,
but even over fields of prime order mostly with additional congruence
conditions on the characteristic. In the present paper we show that the
infinite series of simple groups $E_8(p)$ occur as Galois groups over $\QQ$
for all good primes $p$. \par
Our paper was inspired by the recent result of Zhiwei Yun \cite{Yu} who
showed the Galois realizability of $E_8(p)$ for all sufficiently large
primes $p$, but without giving a bound.  In fact, Yun proved much more
--- he showed that $E_8$ is a motivic Galois group, answering a conjecture
of Serre. \par

Our proof relies on the well-known rigidity criterion of Belyi, Fried, Matzat
and Thompson, but in addition uses deep results mainly of Liebeck and Seitz on
maximal subgroups of
algebraic groups and from Lusztig on the parametrization of irreducible
characters of finite reductive groups, the Springer correspondence and
computations of Green functions. We also require results of Lawther on
fusion of unipotent elements in reductive subgroups. \par

Table~\ref{tab:triples} contains a description of the class triples of
the exceptional groups $G$ of Lie type which we are going to consider.
Here, the involution classes are identified by the structure of their
centralizer in $G$, while the unipotent classes are denoted as in
\cite[\S13.1]{Ca}.

\begin{table}[htbp]
\caption{Candidate classes}   \label{tab:triples}
\[\begin{array}{|r|cc|l|}
\hline
  & G_2(q)&  E_8(q)& \\
\hline
 C_1&  A_1+\tilde A_1&  D_8& \mathrm{involution}\\
 C_2&      \tilde A_1& 4A_1& \mathrm{unipotent}\\
 C_3&             G_2&  E_8& \text{regular unipotent}\\
\hline
\end{array}\]
\end{table}

Our main result is:

\begin{thm} \label{thm:rigidity}
 Let $k$ be an algebraic closure of $\FF_p$ with $p$ prime. Let $G$ be
 either   $G_2(k)$ or $E_8(k)$. Assume that $p$ is good for $G$ (i.e.,
 $p>3$ and if $G=E_8$, $p>5$). Let $C_i$, $1 \le i \le 3$, be the conjugacy
 classes described in Table~\ref{tab:triples}. Let $X$ denote the variety
 of triples in $C_1 \times C_2 \times C_3$ with product $1$. Then $G(k)$
 has a single regular orbit on $X$ and if $(x_1, x_2, x_3) \in X$,
 then $\langle x_1, x_2 \rangle \cong G(\FF_p)$.
\end{thm}

Since $G(k)$ has a single regular orbit on $X$ for $k$ algebraically
closed of  good positive characteristic, it follows that the same is true if
$k$ is an algebraically closed field of
characteristic $0$.
Thus, we obtain a torsor for $G$ and indeed, we can reduce
the question of whether this torsor is trivial to the case of
$C_G(z)$ with $z \in C_3$, a regular unipotent element. Recall that
$C_G(z) \cong k^r$ where $r$ is the rank of $G$ (for $p=0$
or $p$ at least the Coxeter number of $G$). Since torsors over
connected unipotent groups are trivial, it follows that such
triples exist for any field of characteristic $0$.  It is not difficult to show
that some (and so any) such triple generates a Zariski dense
subgroup of $G(k)$ with $k$ algebraically closed of characteristic~$0$.

We can also produce
such triples over $G(\ZZ_p)$ and so show (see Section \ref{sec:char 0}):

\begin{thm} \label{thm:char 0}
 Let $k$ be an algebraically closed field of characteristic $0$.
 Let $G$ be $G_2(k)$ or $E_8(k)$. Let $X$ be the set of
 elements in $C_1 \times C_2 \times C_3$ with product~$1$. For $x \in X$,
 let $\Gamma(x)$ denote the group generated by $x$.
 \begin{enumerate}[\rm(a)]
  \item  For any $x \in X$, $\Gamma(x)$ is Zariski dense in $G(k)$.
  \item  If $k_0$ is a subfield of $k$, then $X(k_0)$ is a single
   $G(k_0)$-orbit (where $G(k_0)$ is the split group over $k_0$).
  \item  Let $m$ be the product of the bad primes for $G$ (i.e., $m=6$ in
   the first  case and $m=30$ for $E_8$) and set $R=\ZZ[1/m]$. There
   exists $x\in X(R)$ such that $\Gamma(x) \le G(R)$ and surjects onto
   $G(R/pR)$ for any good prime $p$. In particular,
   $\Gamma(x)$ is  dense in $G(\ZZ_p)$ for any good prime $p$.
 \end{enumerate}
\end{thm}

Theorem~\ref{thm:rigidity}  implies the following result (answering the
question of Yun for $E_8$).


\begin{thm}   \label{thm:main}
 The finite simple groups $G_2(p)$ ($p\ge5$ prime)
 and $E_8(p)$ ($p\ge7$ prime), occur as Galois groups over $\QQ(t)$,
 and then also infinitely often over $\QQ$. More precisely, the triple
 $(C_1,C_2,C_3)$ of classes as in Table~\ref{tab:triples} is rationally rigid.
\end{thm}

\begin{rems}
(a) The case of $G_2(p)$ ($p\ge5$) had already been shown by
Feit--Fong \cite{FF} (for $p>5$) and Thompson \cite{Th} (for $p=5$).
See also \cite{DR10} \par
(b) The second author has shown that $F_4(p)$ is a
Galois group over $\QQ(t)$ whenever $p\ge5$ has multiplicative order~12
modulo~13, and that $E_8(p)$ is a Galois group over $\QQ(t)$ whenever $p\ge7$
has multiplicative order 15 or 30 modulo~31 (see
\cite[Thm.~II.8.5 and~II.8.10]{MM}).\par
(c)  There are several possible choices of triples for $F_4$ including
one suggested by Yun.   It does seem hard to verify either the character
results or the generation results required for any of these triples.
See the final section.
\end{rems}

It is directly clear from the known classification of unipotent conjugacy
classes (see e.g.~\cite[13.1]{Ca}) that the classes $C_2,C_3$ are rational,
and for class $C_1$ this is obvious. As usual, the proof of rigidity breaks
up into two quite different parts:
showing that all triples $(x_1,x_2,x_3)\in C_1\times C_2\times C_3$ with
product $x_1x_2x_3=1$ do generate $G$, and showing that there is exactly one
such triple modulo $G$-conjugation. The first statement will be shown in
Sections~\ref{sec:F4} and~\ref{sec:E8}, the second in Section~\ref{sec:struct}.

On the way we prove two results which may be of independent interest:
in Theorem~\ref{thm:sym} we note a remarkable symmetry property between the
character table of a finite reductive group and that of its dual, and in
Theorem~\ref{lieprim} we give a short list of possible Lie primitive
subgroups of simple exceptional groups containing a regular unipotent element
(in particular there are none in characteristic larger than $113$).
Combining this with the result of Saxl and Seitz \cite{SS97}, we essentially
know all proper closed subgroups of exceptional groups which contain regular
unipotent elements.

The application of our approach to the other large exceptional groups of
Lie type over prime fields fails due to the fact that for $E_6$ and $E_7$
the finite simple groups are not always the group of fixed points of a
corresponding algebraic group. In particular, the class of regular unipotent
elements in $E_7$ splits into two classes in the finite simple group, which
are never rational over the prime field, when $p>2$. In type $E_6$, again
the class of regular unipotent elements splits, and our approach for
controlling the structure constant does not yield the necessary estimates.
Note that the groups $E_6(p)$ and $\tw2E_6(p)$ are known to occur as Galois
groups for all primes $p\ge5$ which are primitive roots modulo~19
(see \cite[Cor.~II.8.8 and Thm.~II.8.9]{MM}).

Note that, on the other hand almost all families of finite simple groups are
known to occur as Galois groups over suitable (finite) abelian extensions of
$\QQ$, a notable exception being given by the series of Suzuki and Ree groups
in characteristic~2. An overview on most results in this area can be found in
the monograph \cite[Sect.~II.10]{MM}.

We thank Zhiwei Yun for asking the question and for helpful remarks. We thank
Burt Totaro for some suggestions which helped simplify the proof of
Theorem~\ref{thm:char 0}.   We thank Stefan Reiter for observing that
the original triple we considered for $F_4$ could not work by considering
the action on the $26$-dimensional.

\section{Structure Constant } \label{sec:struct}

In this section we give estimates for certain structure constants. For this
we need to collect various results on characters of finite groups of Lie type.
We introduce the following setup, where, in this section only, algebraic groups
are denoted by boldface letters. Let $\bG$ be a connected reductive linear
algebraic group over the algebraic closure of a finite field of
characteristic~$p$, and $F:\bG\rightarrow\bG$ a Steinberg endomorphism with
(finite) group of fixed points $G:=\bG^F$. We assume that all eigenvalues of
$F$ on the character group of an $F$-stable maximal torus of $\bG$ have the
same absolute value, $q$. \par
Let's fix an $F$-stable maximal torus $\bT_0$ of $\bG$. Then the conjugacy
classes of $F$-stable maximal tori of $\bG$ are naturally parametrized by
$F$-conjugacy classes in the Weyl group $W=N_\bG(\bT_0)/\bT_0$ of $\bG$,
that is, by $W$-classes in the coset $W\vhi$, where $\vhi$ denotes the
automorphism of $W$ induced by $F$. If $\bT$ is parametrized by the class of
$w\vhi$, then $\bT$ is said to be \emph{in relative position $w\vhi$} (with
respect to $\bT_0$). Note that in this case $N_G(\bT)/\bT^F\cong C_W(w\vhi)$
(see \cite[Prop.~25.3]{MT}).
\par
For $\bT\le\bG$ an $F$-stable maximal torus and $\theta\in\Irr(\bT^F)$,
Deligne and Lusztig defined a generalized character $R_{\bT,\theta}^\bG$ of
$G$. This character $R_{\bT,\theta}^\bG$ only depends on the $G$-conjugacy
class of $(\bT,\theta)$.

Its values on unipotent elements have the following property (see
\cite[Cor.~7.2.9]{Ca}):

\begin{prop}   \label{prop:valuni}
 Let $u\in G$ be unipotent. Then $R_{\bT,\theta}^\bG(u)$ is independent of
 $\theta$.
\end{prop}

Assume that $\bT$ is in relative position $w\vhi$. Then we write
$Q_{w\vhi}(u):=R_{\bT,\theta}^\bG(u)$ for this common value. In this way each
unipotent element $u\in G$ defines an $F$-class function $Q: W\rightarrow\CC$,
$w\mapsto Q_{w\vhi}(u)$, on $W$, the so-called \emph{Green function}.
By Lusztig's algorithm, the values $Q_{w\vhi}(u)$ are expressible
by polynomials in $q$, at least for good primes $p$, with $q$ in fixed
congruence classes modulo an integer $N_\bG$ only depending on the type
of~$\bG$. For $q$ in a fixed congruence class modulo $N_\bG$, we can thus write
$$Q_{w\vhi}(u)=\sum_{i\ge0}\psi_i^u(w\vhi)\,q^i$$
for suitable class functions $\psi_i^u$ on $W\vhi$, depending on $u$. (In fact,
these are known to be characters of $W\vhi$ when $C_\bG(u)$ is connected.)
We also need to understand the values of Deligne-Lusztig characters on
semisimple elements. First observe the following vanishing result:

\begin{lem}   \label{lem:sum}
 Let $H\le\Irr(\bT^F)$ be a subgroup, and $s\in\bT^F$ semisimple not in the
 kernel of all $\theta\in H$. Then
 $$\sum_{\theta\in H}R_{\bT,\theta}^\bG(s)=0.$$
\end{lem}

\begin{proof}
According to \cite[Lemma~12.16]{DM91} we have
$$R_{\bT,\theta}^\bG(s)\cdot\St(s)=\pm\Ind_{\bT^F}^{\bG^F}(\theta)(s),$$
where $\St$ denotes the Steinberg character of $\bG^F$, and the sign only
depends on $\bT$ and $\bG$, not on $\theta$. Thus
$$\St(s)\sum_{\theta\in H}R_{\bT,\theta}^\bG(s)
  =\pm \sum_{\theta\in H}\Ind_{\bT^F}^{\bG^F}(\theta)(s)
  =\pm \Ind_{\bT^F}^{\bG^F}\big(\sum_{\theta\in H}\theta\big)(s)
  =\pm \Ind_{\bT^F}^{\bG^F}\big(\reg_H)(s) =0,$$
since the regular character $\reg_H$ of $H$ takes value~$0$ on all
non-identity elements. The claim follows since $\St$ does not vanish on
semisimple elements by \cite[Cor.~9.3]{DM91}.
\end{proof}

Now let $\bG^*$ be a group in duality with $\bG$, with corresponding Steinberg
endomorphism also denoted by $F$, and $\bT_0^*\le\bG^*$ an $F$-stable maximal
torus in duality with $\bT_0$. There is a bijection between $G$-classes of
pairs $(\bT,\theta)$ as above, and $G^*:=\bG^{*F}$-classes of pairs
$(\bT^*,t)$, where $\bT^*\le\bG^*$ denotes an $F$-stable maximal torus and
$t\in\bT^{*F}$. Two pairs $(\bT_1,\theta_1)$, $(\bT_2,\theta_2)$ are called
\emph{geometrically conjugate} if under this bijection they correspond
to pairs $(\bT_1^*,t_1)$, $(\bT_2^*,t_2)$ with $G^*$-conjugate elements $t_1$
and $t_2$.

\begin{prop}   \label{prop:valRT}
 Let $s\in G$ be semisimple. Let $(\bT,\theta)$ be in the geometric conjugacy
 class of $t\in G^*$, where $\bT\le \bG$ is in relative position $w\vhi$ with
 respect to a reference torus $\bT_0$ inside $\bC:=C_\bG^\circ(s)$.
 Let $W(s)$ denote the Weyl group of $\bC$, $W(t)$ the Weyl group of
 $C_{\bG^*}^\circ(t)$ and $W_1:=C_{W(t)}(w\vhi)$. Then
 $$R_{\bT,\theta}^\bG(s) =|\bC^F:\bT^F|_{p'}\cdot\sum_{i=1}^r
    |C_{W(t)}(w\vhi):C_{W(t)}(w\vhi)\cap W(s)^{u_i}| \cdot\theta(s^{u_i}),$$
 where $u_1,\ldots,u_r\in W(s)\backslash W/C_{W(t)}(w\vhi)$ are representatives
 for those double cosets such that ${}^{u_i}(w\vhi)\in W(s)$.
\end{prop}

\begin{proof}
By \cite[Cor.~12.4]{DM91} we have
$$R_{\bT,\theta}(s)=\frac{1}{|\bC^F|}
    \sum_{g\in G\atop s\in{}^g\bT^F}R_{{}^g\bT,{}^g\theta}^\bC(s).$$
Now $s\in(^g\bT)^F$ if and only if $^g\bT\subseteq\bC$. Let
$(\bT_1,\theta_1),\ldots,(\bT_r,\theta_r)$ be a system of representatives of
the $C$-classes of $G$-conjugates of $(\bT,\theta)$ with first component
contained in $\bC$. Let
$N_G(\bT,\theta):=\{g\in N_G(\bT)\mid {}^g\theta=\theta\}$ denote the
stabilizer of $(\bT,\theta)$ in $G$, and similarly define
$N_C(\bT_i,\theta_i)$, the stabilizer of $(\bT_i,\theta_i)$ in $C$. Then
using $|N_G(\bT_i,\theta_i)|=|N_G(\bT,\theta)|$ we clearly have
$$R_{\bT,\theta}(s)
  =\sum_{i=1}^r\frac{|N_G(\bT,\theta)|}{|N_C(\bT_i,\theta_i)|}
  R_{\bT_i,\theta_i}^\bC(s).$$
 \par
Let $(\bT_1^*,t_1),\ldots,(\bT_r^*,t_r)$ be a system of representatives of the
$\bC^{*F}$-classes of $G^*$-conju\-gates of $(\bT^*,t)$ with first component
in $\bC^*$. Write $w_i\vhi \in W(t_i)\cap W(s)$ for the relative position
of $\bT_i^*$, and let $u_i\in W(s)\backslash W/C_{W(t)}(w\vhi)$ such that
${}^{u_i}(w\vhi,W(t))=(w_i\vhi,W(t_i))$.
Now $N_G(\bT,\theta)$ is an extension of $\bT^F$ by the subgroup
of $N_G(\bT)/\bT^F$ fixing $\theta$, which under the above duality bijection
is isomorphic to $C_W(w\vhi)\cap W(t)=C_{W(t)}(w\vhi)$.
Similarly $N_C(\bT_i,\theta_i)$ is an extension of $\bT_i^F$ by the
subgroup of $N_C(\bT_i)/\bT_i^F$ fixing $\theta_i$, which is isomorphic to
$$C_{W(t_i)}(w_i\vhi)\cap W(s)={}^{u_i}(C_{W(t)}(w\vhi))\cap W(s)
  \cong C_{W(t)}(w\vhi)\cap W(s)^{u_i}.$$
Since $s$ lies in the centre of $\bC$ we have
$$R_{\bT_i,\theta_i}^\bC(s) =R_{\bT_i,1}^\bC(1)\,\theta_i(s)
                            =|\bC^F:\bT_i^F|_{p'}\cdot\theta_i(s),$$
where the first equality holds by \cite[Prop.~7.5.3]{Ca}. The claim follows
as $|\bT_i^F|=|\bT^F|$.
\end{proof}


We next compute some values of semisimple characters. For any semisimple
element $t\in G^*:=\bG^{*F}$ there is an associated \emph{semisimple character}
$\chi_t$ of $G$, depending only on the $G^*$-class of
$t$, defined as follows: Let $W(t)$ denote the Weyl group of the centralizer
$C_{\bG^*}^\circ(t)$. Let $v\vhi\in W\vhi$ denote the automorphism of $W(t)$
induced by $F$. As explained above, to any pair $(\bT^*,t)$ with
$\bT^*\le C_{\bG^*}^\circ(t)$ an $F$-stable maximal torus there corresponds
by duality a pair $(\bT,\theta)$ consisting of an $F$-stable maximal torus
$\bT\le\bG$ (in duality with $\bT^*$) and $\theta\in\Irr(\bT^F)$, up to
$G$-conjugation. We then write
$R_{\bT^*,t}^\bG:=R_{\bT,\theta}^\bG$. Then by \cite[Def.~14.40]{DM91}
$$\chi_t=\pm\frac{1}{|W(t)|}\sum_{w\in W(t)}R_{\bT_{wv\vhi}^*,t}^\bG\,,$$
where $\bT_{wv\vhi}^*$ denotes an $F$-stable maximal torus in relative position
$wv\vhi$ to $\bT_0^*$, and where the sign only depends on $C_{\bG^*}(t)$.
This semisimple character is irreducible if $C_{\bG^*}(t)$ is connected
(see \cite[Prop.~14.43]{DM91}), so in particular if $\bG$ has
connected center. \par
Thus, $\chi_t(g)$ is nothing else but the multiplicity of the trivial
$F$-class function on $W(t)$ in the $F$-class function on $W(t)$ which maps
an element $w\in W(t)$ to $R_{\bT^*,t}^\bG(g)$, where
$\bT^*\le C_{\bG^*}^\circ(t)$ is an $F$-stable maximal torus in relative
position $wv\vhi$. For unipotent elements this gives:

\begin{cor}   \label{cor:valuni}
 Let $u\in G$ be unipotent, and $Q_{w\vhi}(u)=\sum_{i\ge0}\psi_i^u(w\vhi)\,q^i$
 for $w\in W$ and $q$ in a fixed congruence class modulo $N_\bG$. Then
 $$\chi_t(u)
   =\pm\sum_{i\ge0}\left\langle\psi_i^u|_{W(t)v\vhi},1\right\rangle_{W(t)v\vhi}\,q^i.$$
\end{cor}

\begin{proof}
The above formula for $\chi_t$ and Proposition~\ref{prop:valuni} give
$$\chi_t(u)=\pm\frac{1}{|W(t)|}\sum_{w\in W(t)}Q_{wv\vhi}(u)
  =\pm\sum_{i\ge0}\frac{1}{|W(t)|}\sum_{w\in W(t)}\psi_i^u(w\vhi)\,q^i.$$
As pointed out above the inner term is just the scalar product of the trivial
character with $\psi_i^u$ restricted to the coset $W(t)v\vhi$.
\end{proof}

For example, if $u\in G$ is regular unipotent, then $Q_{w\vhi}(u)=1$ for all
$w\vhi$ by \cite[Prop.~8.4.1]{Ca}, and thus
$\chi_t(u)=\pm\left\langle 1,1\right\rangle_{W(t)v\vhi}=\pm1$.

Let's point out the following remarkable symmetry between the 'semisimple
parts' of character tables of dual groups. For this, we embed $\bG$ into a
connected reductive group $\hat\bG$ with connected center and having the same
derived subgroup as $\bG$, and with an extension $F:\hat\bG\rightarrow\hat\bG$
of $F$ to $\hat\bG$, which is always possible. Then an irreducible character
of $G$ is called semisimple, if it is a constituent of the restriction to $G$
of a semisimple character of $\hat G:={\hat\bG}^F$. By a result of Lusztig,
restriction of irreducible characters from $\hat\bG$ to $\bG$ is multiplicity
free. Note
that all $G$-constituents of a given semisimple character of $\hat G$ take the
same value on all semisimple elements of $G$ since they have the same scalar
product with all Deligne--Lusztig characters, and the characteristic functions
of semisimple conjugacy classes are uniform.

\begin{thm}   \label{thm:sym}
 Let $s\in G$, $t\in G^*$ be semisimple. Then
 $$|C_{\bG^*}(t)^F|_{p'}\,\chi_t(s)=|C_\bG(s)^F|_{p'}\,\chi_s(t).$$
\end{thm}

\begin{proof}
Write $\bC:=C_\bG(s)$ and $\bC':=C_{\bG^*}(t)$.
By \cite[Prop.~7.5.5]{Ca} the characteristic function of the class of $s$
is given by
$$\psi_s=\eps\frac{1}{|\bC^F|_{p}\,|\bC^F|}
    \sum_{(\bT,\theta)\atop s\in\bT}\eps_\bT\,\theta(s)^{-1}R_{\bT,\theta},$$
where the sum ranges over pairs $(\bT,\theta)$ consisting of an $F$-stable
maximal torus $\bT$ of $\bG$ containing $s$ and some $\theta\in\Irr(\bT^F)$,
and where $\eps:=\eps_\bC$ is a sign. (Note that
$|\bC^{\circ F}|_{p}=|\bC^F|_{p}$ always.) Now for any character $\rho$ of $G$
we have $\rho(s)=|\bC^F|\,\langle\psi_s,\rho\rangle$, so that
$$\chi_t(s)=\eps\frac{1}{|\bC^F|_{p}}
  \sum_{(\bT,\theta)\atop s\in\bT}\eps_\bT\,\theta(s)^{-1}
  \langle R_{\bT,\theta},\chi_t\rangle.$$
Now $\langle R_{\bT,\theta},\chi_t\rangle$ is non-zero if and only if
$(\bT,\theta)$ lies in the geometric conjugacy class parametrized by $t$, and
in this case it equals~$\eps':=\eps_{\bC'}$. Indeed, this equality is true for
the group $\hat\bG$ with connected center, and then remains true for $\chi_t$
since the restriction to $G$ is multiplicity free (see \cite[Prop.~5.1]{Lu88}).
So
$$\chi_t(s)=\eps\eps'\frac{1}{|\bC^F|_{p}}
  \sum_{(\bT,\theta)\sim t\atop s\in\bT}\eps_\bT\,\theta(s)^{-1}.$$
Summing over the whole conjugacy class of $s$ we get
$$|G|\chi_t(s)=\eps\eps'\,|\bC^F|_{p'} \sum_{s'\sim s}
  \sum_{(\bT,\theta)\sim t\atop s'\in\bT}\eps_\bT\,\theta(s')^{-1},$$
whence
$$|{\bC'}^F|_{p'}\,\chi_t(s)
  =\eps\eps'\,\frac{|\bC^F|_{p'}\,|{\bC'}^F|_{p'}}{|G|}
  \sum_{s'\sim s}\sum_{(\bT,\theta)\sim t\atop s'\in\bT}\eps_\bT\,\theta(s')^{-1}.
$$
But this last expression on the right hand side is symmetric in $s,t$:
Let $(\bT,\theta)$ be in the geometric conjugacy class of $t$ and $s'\in\bT^F$.
Let $(\bT^*,t')$ be dual to $(\bT,\theta)$ in the sense of
\cite[Prop.~13.13]{DM91}, so $t'\in\bT^{*F}$ which is conjugate to $t$.
Furthermore $s'$ defines an element
$\sigma\in\Irr(\bT^{*F})$, and $s'\in\bT^F$ is equivalent to the fact that
$(\bT^*,t')$ is in the geometric conjugacy class of $s'$, hence of $s$.
By construction $N_G(\bT,\theta)/\bT^F$ equals
$N_{G^*}(\bT^*,t')/\bT^{*F}$, so since $\bT^{*F}$ has the same order as $\bT^F$,
the number of $G$-conjugates of $(\bT,\theta)$ and of $G^*$-conjugates of
$(\bT^*,t')$ agree. Thus instead of summing over
triples $(s',\bT,\theta)$ we may sum over the dual triples $(t',\bT^*,\sigma)$,
with $t'\sim t$, and $\sigma'\in\Irr(\bT^{*F})$, so that
$\theta(s')=\sigma(t')$.
\end{proof}

\begin{rem}
For every semisimple element $t\in G^*:=\bG^{*F}$ there is also a \emph{regular
character}
$$\chi_t^\reg=\pm\frac{1}{|W(t)|}\sum_{w\in W(t)}\eps_{\bT_w^*}
  R_{\bT_w^*,t}^\bG$$
of $G$ (see \cite[Def.~14.40]{DM91}), where $\bT_w^*$ denotes an $F$-stable
maximal torus in relative position $wv\vhi$ to $\bT_0^*$, $\eps_{\bT_w^*}$ is
a sign, and where the global sign only depends on $C_{\bG^*}(t)$.
This regular character is irreducible if $C_{\bG^*}(t)$ is connected
(see \cite[Prop.~14.43]{DM91}), so in particular if $\bG$ has
connected center.
Entirely analogously to Theorem~\ref{thm:sym} one can show that
$$|C_{\bG^*}(t)^F|_{p'}\,\chi_t^\reg(s)=|C_\bG(s)^F|_{p'}\,\chi_s^\reg(t)$$
for all semisimple $s\in G$, $t\in G^*$.
\end{rem}

\begin{thm}   \label{thm:triples}
 Let $G=G(q)$ be one of the finite simple groups of Lie type in
 Table~\ref{tab:triples}, with $q=p^f$ a power of a good prime $p$ for $G$.
 Let $x\in G$ be an involution, $y\in G$ a unipotent element as indicated
 in the table, and $z$ a regular unipotent element. Set
 $$f(q):=  \sum_{1\ne\chi\in\Irr(G)}\frac{\chi(x)\chi(y)\chi(z)}{\chi(1)}.$$
 Then $f(q)$ is  a rational function in $q$, for all $q$ in a
 fixed residue class modulo a sufficient large integer only depending on
 the type of $G$ and $|f(q)| < 1$ for $q$ sufficiently large.
\end{thm}

\begin{proof}
Let $\bG$ denote a simple algebraic group of exceptional type defined
over $\FF_q$ and $F:\bG\rightarrow\bG$ a Steinberg endomorphism so that
$G=\bG^F$. \par
In order to investigate the sum, we make use of Lusztig's theory of characters.
We argue for all $q$ in a fixed congruence class modulo $N_\bG$ (see above).
First of all, since we assume that $p$ is a good prime for $G$, it follows
that only the semisimple characters of $G$ do not vanish on the class $[z]$
of regular unipotent elements, and the semisimple characters take value
$\pm1$ on that class (see \cite[Cor.~8.3.6]{Ca}). Since $\bG$ has connected
center, the dual group $\bG^*$ is of simply connected type, hence all
semisimple elements of $\bG^*$ have connected centralizer. Thus, the
semisimple characters of $G$ are in one-to-one correspondence with the
$F$-stable semisimple conjugacy classes of $\bG^*$, and we write
$\chi_t$ for the semisimple character indexed by (the class of) a semisimple
element $t\in\Gss$. \par
Let's say that two semisimple elements of $\bG^{*F}$ are equivalent if their
centralizers in $\bG^{*F}$ are conjugate. Then it is known that the number of
equivalence classes is bounded independently of $q$, and can be computed purely
combinatorially from the root datum of $\bG$ (see e.g. \cite[Cor.~14.3]{MT}).
Now note that if $t_1,t_2\in\Gss$ are equivalent,
then $\chi_{t_1}$ and $\chi_{t_2}$ agree on all unipotent elements, since
by the formula in Corollary~\ref{cor:valuni} the value of $\chi_t$ only
depends on $C_{\bG^*}(t)$. Thus in order to prove the claim it suffices to show
that for each of the finitely many equivalence classes $A$ of semisimple
elements in $\Gss$ up to conjugation we have
$$\Big|\frac{\chi(y)\chi(z)}{\chi(1)}\sum_{t\in A}\chi_t(x)\Big|=O(q^{-1}),$$
where $\chi(u):=\chi_t(u)$ denotes the common values of all $\chi_t$, $t\in A$,
on a unipotent element $u$.
For this, we compute the degree $d_u(A)$ in $q$ of the
rational function $\frac{\chi(y)\chi(z)}{\chi(1)}$ explicitly from the known
values of the Green functions (see Lusztig \cite{LuV} and Spaltenstein
\cite{Spa}) using Corollary~\ref{cor:valuni}. This is a purely mechanical
computation with reflection cosets inside the Weyl group of $\bG$ and can be
done in \Chevie{} \cite{MChev} for example.
\par
It remains to control the sums $\sum_{t\in A}\chi_t(x)$, for $A$ an equivalence
class of semisimple elements (up to conjugation). Let's fix $t_0\in A$ and
set $\bC_A=C_{\bG^*}^\circ(t_0)$. By duality, we may interpret $s$ as a linear
character (of order~2) on all maximal tori of $\bC_A$. First of all, since
$C_\bG(s)$ has finitely many classes of maximal tori, and each torus only
contains finitely many involutions (conjugate to $x$), there are only finitely
many possibilities for the values $\{\chi_t(s)\mid t\in A\}$, as a
polynomial in $q$. Using \Chevie{} again, we can calculate the maximal
degree $d_s(A)$ in $q$ of any such polynomial from
Theorem~\ref{thm:sym}. Secondly, the number of elements
in $A$ is a polynomial in $q$ of degree $d(A):=\dim Z(\bC_A)$ since the set
$$\{t\in Z(\bC_A)\mid C_{\bG^*}^\circ(t)=\bC_A\}$$
is dense in $Z(\bC_A)$ (see \cite[Ex.~20.11]{MT}). But whenever there is
some $t\in A$ not in the kernel of $s$, then
$\sum_{t\in Z(\bC_A)^F}\chi_t(s)=0$ by Lemma~\ref{lem:sum}. So
$\sum_{t\in A}\chi_t(s)=-\sum_{t\in Z(\bC_A)^F\setminus A}\chi_t(s)$,
and the number of elements in $Z(\bC_A)^F\setminus A$ is given by a polynomial
in $q$ of degree strictly smaller than $d(A)$. \par
Explicit computation now shows that for all equivalence classes $A$ of
semisimple elements in $G^*$, the sum of the degrees
$d_u(A)+d_s(A)+d(A)$, respectively $d_u(A)+d_s(A)+d(A)-1$ in the case that
there is some $t\in A$ not in the kernel of $s$, is smaller than~$0$,
whence the claim.
\end{proof}

\begin{rem}
 In fact, using information on subgroups containing regular unipotent
 elements, we will see that $f(q)=0$ for all $q = p^a$ with $p$ good,
 see Theorem \ref{thm:rigidity}.
\end{rem}

\begin{rem}
A quick computation with the generic character table gives that for
$G=\tw3D_4(p^f)$, $p\ge3$, the normalized structure constant of
$(C_1,C_2,C_3)$ with $C_1$ the class of involutions, $C_2$ the class of
unipotent elements of type $3A_1$, and $C_3$ the class of regular unipotent
elements equals~1.
But since all three classes intersect $G_2(p)$ non-trivially, and the
structure constant there equals~1 as well, these triples only generate
$G_2(p)$ respectively $\SL_2(8)\cong{}^2G_2(3)'$ for $p=3$.
\end{rem}

\section{Lie Primitive Subgroups Containing Regular Unipotent Elements}

Let $G$ be a simple algebraic group (of adjoint type) over
an algebraically closed field of characteristic $p \ge 0$.   We want
to consider the closed subgroups of $G$ containing a regular unipotent
element of $G$.   The maximal closed subgroups of positive dimension
containing a regular unipotent element are classified in
\cite[Thm.~A]{SS97}.   Of course, subfield subgroups and parabolic subgroups
contain regular unipotent elements.  Thus, we focus on the
Lie primitive subgroups (those finite groups which
do not contain a subgroup of the form $O^{p'}(G^F)$ where $F$
is some Frobenius endomorphism of $G$ and are not contained in any proper
closed positive dimensional subgroup). Note that if $p=0$, unipotent elements
have infinite order and  so any closed subgroup containing a regular unipotent
element has positive dimension. So we assume that $p > 0$.

We record the following well known lemma.

\begin{lem}   \label{lem:adjoint}
 Let $G$ be a simple algebraic group of rank $r$ over an algebraically closed
 field. Let $W = \Lie(G)$ denote the adjoint module for $G$.
 If $w \in W$, then the stabilizer of $w$ in $G$ has dimension at least $r$.
\end{lem}

\begin{proof}
Since the condition on dimension is an open condition, it suffices to prove
this for $w$ corresponding to a semisimple regular element in $W$. In this
case, the stabilizer of $w$ in $G$ is a maximal torus which has rank~$r$.
\end{proof}

We only consider exceptional groups here. One could prove a similar result
for the classical groups using \cite{GPPS} and \cite{Di}. The following
well-known result on the orders of regular unipotent elements
in the exceptional groups will be used throughout the subsequent proof.
This can be read off from the tables in \cite{Law95}.

\begin{lem}   \label{lem:orders}
 Let $G$ be an exceptional group of Lie type in characteristic $p > 0$ with
 Coxeter number $h$. Then the order of regular unipotent elements of $G$ is as
 given in Table~\ref{tab:orders}.
 \end{lem}

\begin{table}[htbp]
\caption{Orders of regular unipotent elements}   \label{tab:orders}
\[\begin{array}{|r|ccccc|}
\hline
 G& p=2& p=3& p=5& 5<p<h& h\le p\\
\hline
     G_2& \ \,8& \ 9& \ 25& p^2& p\\
 F_4,E_6&  16&  27& \ 25& p^2& p\\
     E_7&  32&  27& \ 25& p^2& p\\
     E_8&  32&  81&  125& p^2& p\\
\hline
\end{array}\]
\end{table}

We will give all possibilities for maximal
Lie primitive subgroups of simple exceptional groups containing a regular
unipotent element (we are certainly not classifying all cases up
to conjugacy nor are we claiming that all cases actually do occur --- although
one can show that several of the cases do occur).

We deal with $G_2(k)$ first. In this case, all maximal subgroups of the
associated finite groups are known \cite{Coop, PBK}  and so it is a simple
matter to deduce:

\begin{thm}   \label{g2}
 Let $G=G_2(k)$ with $k$ algebraically closed of characteristic $p > 0$.
 Suppose that $M$ is a maximal Lie primitive subgroup of $G$ containing a
 regular unipotent element.  Then one of the following holds:
 \begin{enumerate}[\rm(1)]
  \item  $p=2$ and $M=J_2$;
  \item  $p=7$ and $M= 2^3.\PSL_3(2)$, $G_2(2)$ or $\PSL_2(13)$; or
  \item  $p=11$ and $M= J_1$.
 \end{enumerate}
\end{thm}

Note that in the previous theorem, each of the possibilities does contain
a regular unipotent element.  In (1), this follows by observing that since
$G_2(k) < \Sp_6(k)$, any element of order $8$ has a single Jordan block and
so is regular unipotent in $G$.  In all possibilities in (2), $M$ acts
irreducibly on the $7$ dimensional module $V$ for $G$ and has a Sylow
$7$-subgroup of order $7$.  Thus, $V$ is a projective $M$-module, whence an
element of order $7$ has a single Jordan block of size~$7$.  The only
unipotent elements of $G$ having a single Jordan block on $V$ are the regular
unipotent elements \cite{Law95}.
In (3), we note that $M$ contains $\PSL_2(11)$ which acts irreducibly
and so elements of order $11$ have a single Jordan block.

We now consider $G$ of type $F_4$, $E_6$, $E_7$ or $E_8$; here we
let $t(G)$ be defined as in \cite{LS03}.

\begin{thm}   \label{lieprim}
 Let $G$ be a simple algebraic group over an algebraically closed field $k$
 of characteristic $p > 0$. Assume moreover, that $G$ is exceptional of rank
 at least $4$. Suppose that $M$ is a maximal Lie primitive subgroup of $G$
 containing a regular unipotent element.
 \begin{enumerate}[\rm(a)]
  \item  If $G=F_4(k)$ then one of the following holds:
   \begin{enumerate}[\rm(1)]
    \item  $p=2$ and $F^*(M)= \PSL_3(16)$, $\PSU_3(16)$ or $\PSL_2(17)$;
    \item  $p=13$ and $M = 3^{3}:\SL_3(3)$ or $F^*(M)=\PSL_2(25)$,
     $\PSL_2(27)$ or $\tw3D_4(2)$; or
    \item  $M=\PSL_2(p)$ with $13 \le p \le 43$.
   \end{enumerate}
  \item If $G=E_6(k)$ then one of the following holds:
   \begin{enumerate}[\rm(1)]
    \item  $p=2$ and $F^*(M)=\PSL_3(16)$, $\PSU_3(16)$ or $Fi_{22}$;
    \item  $p=13$ and $M = 3^{3+3}:\SL_3(3)$ or $F^*(M)=\tw2F_4(2)'$; or
    \item  $M=\PSL_2(p)$ with $13 \le p \le 43$.
   \end{enumerate}
  \item If $G=E_7(k)$ then one of the following holds:
   \begin{enumerate}[\rm(1)]
    \item  $p=19$ and $F^*(M) = \PSU_3(8)$ or $\PSL_2(37)$; or
    \item  $M=\PSL_2(p)$ with $19  \le p \le 67$.
   \end{enumerate}
  \item If $G=E_8(k)$ then one of the following holds:
   \begin{enumerate}[\rm(1)]
    \item  $p=2$ and $F^*(M) = \PSL_2(31)$;
    \item  $p=7$ and $F^*(M)=\PSp_8(7)$ or $\Om_9(7)$;
    \item  $p=31$ and $M = 2^{5+10}.\SL_5(2)$ or $5^3.\SL_3(5)$, or
       $F^*(M) = \PSL_2(32)$, $\PSL_2(61)$ or $\PSL_3(5)$; or
    \item  $M=\PSL_2(p)$ with $31 \le p  \le 113$.
   \end{enumerate}
 \end{enumerate}
\end{thm}

\begin{proof}
Let $G$ be a simple exceptional algebraic group over $k$ of rank at least $4$.
Let $M$ be a maximal Lie primitive subgroup of $G$  (i.e., $M$ is Lie
primitive, not a subfield group, and is not contained in any finite subgroup
of $G$ other than subfield groups). We split the analysis into various cases.
The possibilities for $M$ are essentially listed in \cite[Thm.~8]{LS03}.
See also \cite{CLSS, LS99}.


\medskip\noindent
{\bf Case 1.} $M$ has a normal elementary abelian $r$-subgroup (with $r\ne p$).

By \cite{CLSS}, this implies that one of the following holds:
\begin{enumerate}
\item  $p \ne 3$,  $G= F_4(k)$ with $M \cong  3^{3}:\SL_3(3)$;
\item  $p \ne 3$,  $G= E_6(k)$ with $M \cong 3^{3+3}:\SL_3(3)$;
\item  $p \ne 2$,  $G= E_8(k)$ with $M \cong 2^{5+10}.\SL_5(2)$; or
\item  $p \ne 5$,  $G= E_8(k)$ with $M \cong 5^3.\SL_3(5)$.
\end{enumerate}
By considering the order of a regular unipotent element,
 we see that the only possibilities
are $p=13$ in~(1) or~(2) and $p=31$ in~(3) or~(4).

\medskip\noindent
{\bf Case 2.} $F(M)=1$ but $M$ is not almost simple.

By \cite{LS03}, the only possibility is that $G=E_8(k)$ and
$M \cong (A_5 \times A_6).2^2$. By considering the exponent of $M$ compared
to the order of a regular unipotent element, we see that $M$ contains no
regular unipotent elements.

\medskip\noindent
{\bf Case 3.}  $F^*(M)$ is a simple group of Lie type in characteristic $p$
of rank~1.

We first deal with the case that $F^*(M)=\PSL_2(p^a)$. Suppose that a regular
unipotent element of $G$ has order $p^b$ with $b > 1$.  Then
$M$ must involve a field automorphism of order $p^{b-1}$, whence
$a \ge p^{b-1}$ and it follows that $p^a > (2,p-1) t(G)$, whence this case does
not occur by \cite{Law12}.

Thus, we may assume that the regular unipotent element has order $p$ which
gives us the lower bound for $p$ in the result.  It follows by
\cite[Thms.~1.1 and 1.2]{ST2} that $a=1$ and $M=\PSL_2(p)$. The upper bound
for $p$ follows by \cite[Thm.~2]{ST1} (see also \cite[Thm.~29.11]{MT}).


If $p=2$ and $F^*(M) =\tw2B_2(2^{2a+1})$, $a \ge 1$, then the exponent of the
Sylow $2$-subgroup of $M$ is $4$ and so $M$ will not contain a regular
unipotent element.

If $p=3$ and $F^*(M)={^2}G_2(3^{2a+1})'$, then the exponent of a Sylow
$3$-subgroup of $F^*(M)$ is $9$. Thus, there must be a field automorphism of
order $3$ in $M$ (or of order~$9$ when $G=E_8(k)$).  It follows that
$3^{2a+1} > 2t(G)$ unless $2a+1 = 3$ and $G=F_4, E_6$ or $E_7$.
Thus, $M= {^2}G_2(27).3$.  Let $V$ be the adjoint module for $G$.  The only
irreducible representations of $M$ in characteristic $3$ of dimension
at most $\dim V$ are the trivial module,  a module of dimension $21$
or if $G=E_7$ a module of dimension $81$.  It follows that  by noting
that $\dim H^1(M,W) \le 1$ for any of the possible modules $W$ occurring
as composition factors of $V$ \cite{sin}.  It follows easily that $M$
has fixed points on $V$, whence by Lemma \ref{lem:adjoint} that $M$
is contained in a positive dimensional subgroup of $G$, so this
case does not occur.

\medskip\noindent
{\bf Case 4.}  $F^*(M)$ is a simple group of Lie type in characteristic $p$,
of  rank $r$ at least $2$.

Generically in this case, $M$ will be contained in a positive dimensional
subgroup. However if the rank and field size are small, there are some
possibilities left open.

It follows by \cite[Thm.~8]{LS03} that $r \le 2s$ where $s$ is the rank of $G$.
Similarly it follows that either $q \le 9$ or $F^*(M)=\PSU_3(16)$ or
$\PSL_3(16)$.

The cases to deal with are therefore:

$F^*(M)=\PSU_3(2^a)$, $1<a\le 4$ or $\PSL_3(2^a)$, $1\le a\le 4$ with $p=2$.
In this case, the exponent of a Sylow $p$-subgroup of $M$ is at most $8$
unless $2^a=16$ and so $M$ contains no regular unipotent elements.
If $2^a=16$, the same argument rules out $E_7(k)$ and $E_8(k)$.

Next suppose that $F^*(M)= \PSL_3(q)$ or $\PSU_3(q)$ with $q = 3, 5, 7$ or~$9$.
The exponent rules out the possibility that $M$ contains a regular unipotent
element.

Next consider the case that $F^*(M) = \PSp_4(q),\PSL_4(q)$ or $\PSU_4(q)$ with
$q = p^a \le 9$.
If $p$ is odd, then the exponent of a Sylow $p$-subgroup of $M$ is either
$p$ or $9$, a contradiction.   If $q$ is even, the exponent of $M$
is at most $8$, also a contradiction.

Next suppose that $F^*(M)=\PSp_6(q)$ or $\Om_7(q)$ with $q = p^a\le9$.
Again, it follows that the exponent of a Sylow $p$-subgroup of $M$ is smaller
than the order of a regular unipotent element of $G$.

The remaining cases are when $M$ has rank $4$ and is defined over a field of
size $q=p^a \le 9$ and so we may assume that $G=E_8(k)$. If $p=2$, then the
exponent of a Sylow $2$-subgroup of $M$ is at most $16$, which is too small by
Table~\ref{tab:orders}.  Similarly, if $p=3$ or $5$, then
the exponent of a Sylow $p$-subgroup of $M$ is at most $p^2$, again too small.
The remaining possibility is that $p  = q =7$, whence $F^*(M)=\PSp_8(7)$ or
$F^*(M)=\Om_9(7)$ by Table~\ref{tab:orders}.

\medskip\noindent
{\bf Case 5.} $F^*(M)$ is a simple group not of Lie type in characteristic $p$.

We can eliminate almost all of these by comparing the order of a regular
unipotent element to the exponent of the possibilities for $M$ given in
\cite[Thm.~8]{LS03}. Moreover, the element of the right order
must have centralizer a $p$-subgroup (since this is true for regular
unipotent elements). The possibilities remaining are given in the theorem.
\end{proof}

\begin{rem}
One can show that some of the subgroups listed in Theorem~\ref{lieprim} are Lie
primitive and do contain regular unipotent elements.
The possibilities with $M \cong \PSL_2(p)$ given above likely do not occur
(indeed this follows by Magaard's thesis \cite{Ma} for $F_4$ and
by unpublished work of Aschbacher \cite{asch} for $E_6$).

Note in particular that if $p > 113$, then there are no Lie primitive
subgroups containing a regular unipotent element (and likely this is true
for $p > 31$).
\end{rem}

Note the following corollary.

\begin{cor}   \label{cor:e8prim}
 Let $G=E_8(k)$ over an algebraically closed field $k$ of characteristic
 $p > 5$. Suppose that $x$ is an involution in $G$,  $y$ is in the
 conjugacy class $4A_1$ and $z$ is a regular unipotent element with $xyz=1$.
 \begin{enumerate}
 \item If $p > 7$, then $\langle x, y \rangle$ is not contained in a Lie
  primitive subgroup.
 \item If $p=7$ and $\langle x, y \rangle$ is contained in a Lie primitive
  subgroup, then $\langle x, y \rangle$ is contained in a proper closed
  subgroup of $G$ of positive dimension.
 \end{enumerate}
\end{cor}

\begin{proof}
We use the previous result. Suppose that $H:=\langle x,y \rangle \le M$ with
$M$ a maximal Lie primitive subgroup of $G$. Consider the possibilities for
$M$ in Theorem~\ref{lieprim}(4) with $p > 5$.

If $M = \PSL_2(p)$, then $M$ intersects a unique conjugacy class of
elements of order $p$ in $G$, a contradiction.
Similarly, in  the cases with $p=31$, a Sylow $p$-subgroup of $M$ is cyclic, a
contradiction.

The only cases remaining are with $p=7$ and $F^*(M)=\PSp_8(7)$ or
$F^*(M)= \Om_9(7)$.   Thus (i) holds.    So consider the remaining
case with $p=7$ and assume that $H$ is not contained in a proper closed
positive dimensional subgroup
of $G$.  It follows that $H$ is not contained in a parabolic subgroup of
$M$ either (for then $H$ would normalize a unipotent subgroup and so
be contained in a parabolic subgroup of $G$ as well).

Since $H$ is generated by unipotent elements, it follows that $H \le F^*(M)$.
Since there are no maximal subgroups of $F^*(M)$ other than parabolic
subgroups containing a regular unipotent element, it follows that $H=F^*(M)$.

Note that $y$ cannot act quadratically
on the natural module for $H$ (because then $x$ and $y$ do not
generate an irreducible subgroup). If $H=\Om_9(7)$, then similarly,
we see that $y$ is not a short root element. It follows by the main results of
\cite{Su} that on any irreducible module other than
the natural or the trivial module  for $H$ in characteristic $7$,
$y$ has a Jordan block of size at least $5$.  However,
$y$ has all Jordan blocks of size at most $4$ on the adjoint module $W$
for $E_8$.    It follows that all composition factors are trivial in case
$H = \PSp_8(7)$ (since the natural module is not a module for the
simple group).   In case,  $H=\Om_9(7)$, since
$H^1(H,V)=\mathrm{Ext}_H^1(V,V)$ with $V$ the natural module, it
follows that $W$ is a semisimple $H$-module and $H$ must have
a fixed point on $W$ (since $248$ is not a multiple of $9$).
However, the stabilizer of a point of $W$ has dimension at least $8$
by Lemma~\ref{lem:adjoint} and so $H$ is contained in a positive dimensional
proper closed subgroup, a contradiction.

This completes the proof.
\end{proof}

We have similar results for the other groups.  In particular, we see that:

\begin{cor}   \label{cor:f4prim}
 Let $G=F_4(k)$ over an algebraically closed field $k$ of characteristic $p>3$.
 Suppose that $x$ is an involution in $G$,
 $y$ is a unipotent element in the class $A_1 + \tilde{A_1}$ and $z$ is a
 regular unipotent element.   If  $xyz=1$, then $H = \langle x, y \rangle$
 is not contained in a Lie primitive subgroup of $F_4(k)$.
\end{cor}

\begin{proof}
By assumption $H \le M$ for some $M$ as given in Theorem~\ref{lieprim}(1).
However, in all cases with $p > 3$, $M$ only intersects a single unipotent
class of $G$.
\end{proof}

\section{Some Nonexistence Results}

\begin{lem}   \label{sl}
 Let $k$ be a field of characteristic $p \ne 2$. Let $G=\SL_n(k)=\SL(V)$.
 Assume that $x \in G$ is an involution, $y \in G$ is a unipotent element
 with quadratic minimal polynomial and $z \in G$ is a regular unipotent
 element. Then $xyz\ne1$.
\end{lem}

\begin{proof}
If $n=2$, the only involution is central and the result is clear. If $n=3$,
we see that $x$ and $y$ have a common eigenvector $v$ with $xv =-v$.
Thus, $xy$ is not unipotent.

So assume that $n \ge 4$.   If $x$ and $y$ have a common eigenvector,
the result follows by induction.  Thus, $n=2m$ and the fixed spaces of
$x$ and $y$ on $V$ each have dimension $m$.   Thus, choosing an appropriate
basis for $V$, we may assume that:
$$
  x = \begin{pmatrix}   I_m  &  J \\   0  & -I_m \\ \end{pmatrix}
  \quad\text{and}\quad
  y  =\begin{pmatrix}   I_m  &  0  \\   I_m  & I_m \\ \end{pmatrix},
$$
where $J$ is in Jordan canonical form.  If $J$ has more than
$1$ block, then $V = V_1 \oplus V_2$ with $V_i$ invariant under
$\langle x, y \rangle$, whence $xy$ is certainly not regular unipotent.
Note that
$$
xy - I_n = \begin{pmatrix}   J  &  J \\  -I_m  & - 2I_m  \\ \end{pmatrix}.
$$
If $J$ is not nilpotent, then we see that $xy - I_n$ is invertible, whence
$xy$ is not unipotent (indeed has no eigenvalue $1$).
If $J$ is nilpotent, we see that $-2$ is an eigenvalue and so
again $xy$ is not unipotent.
\end{proof}

By viewing $\SO_{2m}(k)$ inside $\SL_{2m}(k)$ and starting with $m=2$,
essentially the same proof yields:

\begin{lem}   \label{so}
 Let $G= \SO_{2m}(k)$, $m \ge 2$, with $k$ of characteristic $p \ne 2$. Assume
 that $x \in G$ is an involution, $y \in G$ is a unipotent element with
 quadratic minimal polynomial and $z \in G$ is a regular unipotent element.
 Then $xyz \ne 1$.
\end{lem}

We will also need to deal with one case where the unipotent element has
a Jordan block of size $3$.

\begin{lem}   \label{spin14}
 Let $k$ be a field of characteristic $p \ne 2$.
 Let $G=\Spin_{14}(k)$. Let $V$ be the natural $14$-dimensional module for $G$.
 If $x \in G$ is an involution,  $y \in G$ is unipotent with $\dim C_V(y)\ge8$
 and $z \in G$ is a regular unipotent element, then $xyz \ne 1$.
\end{lem}

\begin{proof}
Since $x$ is an involution in $G$, the $-1$ eigenspace
of $x$ on $V$ either has dimension at least $8$ or has dimension
at most $4$.   If this  dimension is at least $8$, then $x$ and $y$
have a common eigenvector $v$ with $xv = - v$, whence $xy$
is not unipotent.  If this dimension is at most $4$, then
$2 = \dim C_V(z)  \ge \dim C_V(x) \cap C_V(y) \ge 4$, a contradiction.
\end{proof}

\section{Rigidity for $E_8$}   \label{sec:E8}

Let $p$ be a prime with $p \ge 7$. Let $G=E_8(k)$ with $k$ the algebraic
closure of the prime field $\FF_p$.  Let $C_1$ be the conjugacy class of
involutions with centralizer $D_8(k)$, $C_2$ the unipotent conjugacy class
$4A_1$ and $C_3$ the class of regular unipotent elements  in $G$.
Observe that since $p > 5$,  the centralizers of elements in these classes
are connected and so $C_i \cap E_8(q)$ is a single conjugacy class.

\begin{thm}   \label{e8gen}
 Let $G=E_8(k)$ with $k$ the algebraic closure of $\FF_p$ with  $p > 5$.
 If $(x,y,z)\in(C_1,C_2,C_3)$ with $xyz=1$, then $H:=\langle x, y \rangle
 \cong E_8(q)$ with $q=p^a$ for some $a$.
\end{thm}

\begin{proof}
Assume that $H$ does not contain a conjugate of $E_8(p)$.
By Corollary~\ref{cor:e8prim}, it follows that $H$ is contained in a maximal
closed subgroup of $G$ of positive dimension.
By \cite[Thm.~A]{SS97}, the only reductive such subgroup
would be isomorphic to $A_1(k)$.  Since $A_1(k)$ has a unique conjugacy
class of unipotent elements, it cannot intersect both $y^G$ and $z^G$.

The remaining possibility is that $H \le P$ where $P$ is a maximal parabolic
subgroup.  Write $P=QL$ where $L$ is a Levi subgroup of $P$ and $Q$ is
the unipotent radical. Set $S=[L,L]$.  Since $y$ and $z$ are unipotent,
$H \le [P,P]=QS$.

Write $x=x_1x_2$, $y = y_1y_2$ and $z=z_1z_2$ where $x_1, y_1, z_1 \in Q$
and $x_2,y_2,z_2 \in L$. Note that $z_2$ is a regular unipotent element
in $L$.

It follows by \cite{Law95} that if $S_1$ is a direct factor of $S$
of type $A$, then the projection of $y_2$ in $S_1$ is a quadratic unipotent
element (because of the Jordan block structure on
the adjoint module).   Applying Lemma \ref{sl} gives a contradiction if
$S_1 \cong A_j(k)$ with $j \ge 2$.

Thus, $S \cong E_7(k), \Spin_{14}(k)$ or $A_1(k)E_6(k)$. If $S=\Spin_{14}(k)$,
it follows by \cite{Law95} that $y_2$ is either a quadratic unipotent
element or has one Jordan block of size $3$ and all other Jordan blocks of
size at most $2$.  Now Lemma \ref{spin14} gives a contradiction.

So we see that either $S \cong E_7(k)$ or $A_1(k)E_6(k)$.
Suppose that $S = A_1(k)E_6(k)$. It then follows that $x_1$ must be trivial
in $A_1(k)$ and so in that case $H$ is contained in a (non-maximal) parabolic
subgroup $P_1=Q_1E_6(k)$ with unipotent radical $Q_1$ and semisimple part
$E_6(k)$. Let $H_0$ be the projection of $H$ in $E_6(k)$. By
\cite{Law95}, it follows
that $y_2$ will be in one of the classes $3A_1, 2A_1, A_1$ or $1$.
Let $J=E_6(k) \le S$. Let $V$ be the Lie algebra of $J$.
Then $\dim [x_2, V] \le 40$ and $\dim [y_2, V] \le  40$. It follows
that $\dim [x_2, V] + \dim [y_2, V] + \dim [z_2,V] \le 152$.
By Scott's Lemma \cite{scott}, $H_0$ has a fixed point on $V$ (since $V$
is a self dual module). Thus $H_0$ is contained in a positive dimensional
maximal closed subgroup $M$ of $J$ by Lemma~\ref{lem:adjoint}.
By \cite{SS97}, this implies that $M$ is either parabolic
or $M \cong F_4(k)$.  However, $F_4(k)$ has no fixed points on $V$,
so this case cannot occur.  So  $H_0$ is contained in a proper parabolic
subgroup of $QE_6(k)$, whence $H$ is contained in at least $3$ distinct
maximal parabolic subgroups.  However, this contradicts the fact that there
are at most $2$ maximal parabolic subgroups containing our triple
(i.e., the $E_7(k)$ parabolic or the $A_1(k)E_6(k)$ parabolic).

It follows that $H$ is contained only in an $E_7(k)$ parabolic.
Since $E_7(k)$ in $E_8(k)$ is simply connected, it follows
that $x_2$ has centralizer $D_6(k)A_1(k)$.  Arguing as above,
we see that  $y_2$ is in the closure
of $4A_1$ (in $E_7(k)$).   Let $W$ denote the Lie algebra of $E_7(k)$.
It follows that $\dim [x_2, W] + \dim [y_2,W] + \dim [z_2,W] <  2 \dim W$,
whence $H$  has a fixed point acting on $W$ and so $QH$ is contained is a
positive dimensional subgroup of $P$. By \cite{SS97}, either $H$ is contained
in a proper parabolic subgroup of $P$ or $H$ is contained
in $X:=A_1(k) \wr \PSL_2(7)$ with $p=7$. In the first case, $H$ would be
contained in another maximal parabolic subgroup (not of type $E_7$), a
contradiction.  In the latter case, we note that a regular unipotent
element of $G$ has order $49$ and in particular is not contained in $F^*(X)$.
Note that $y$ has order $7$ and all Jordan blocks of $y$ on the Lie algebra
of $E_8$ have size at most $4$ \cite{Law95}.  However, any unipotent element
of $X$ outside $F^*(X)$ has a Jordan block of size $7$ on any module
where $F^*(X)$ acts nontrivially. This contradiction completes the proof.
\end{proof}

\begin{thm}   \label{rigid1}
 The subvariety $X=\{(x,y,z)\in C_1 \times C_2 \times C_3\mid xyz=1\}$ is
 a regular $G$-orbit and if $(x,y,z) \in X$, then $\langle x, y \rangle$ is
 a conjugate of $E_8(p)$. In particular, $(C_i \cap E_8(p)\mid 1 \le i \le 3)$
 is a rationally rigid triple.
\end{thm}

\begin{proof}
By Theorem~\ref{thm:triples}, we know that $X$ is an irreducible variety
of dimension equal to $\dim G$.  By Theorem \ref{e8gen}, the centralizer of
any triple in $X$ is trivial, whence $X$ is a single $G$-orbit.
Note that $X$ is defined over $\FF_p$.
So by Lang's theorem, the $\FF_p$-points of $X$ form a single $E_8(p)$-orbit.
Applying Theorem~\ref{e8gen} once again, we see that any triple generates
a subgroup isomorphic to $E_8(p)$.
\end{proof}

An application of the rigidity criterion (see e.g. \cite[Thm.~I.4.8]{MM}) now
completes the proof of Theorem~\ref{thm:main}.

\section{Characteristic Zero}  \label{sec:char 0}

\subsection{Fields}  \label{ss:fields}

Let $G$ be a simple algebraic group of type $G_2$ or $E_8$ over an
algebraically closed field $k$ of characteristic $0$ with the conjugacy
classes $C_i$ defined as in Table~\ref{tab:triples}.

Note that we have the following result in characteristic $0$ (by essentially
the same proof as for Theorem~\ref{rigid1} or) by noting that the number of
orbits in characteristic $0$
is the same as in characteristic $p$ for sufficiently large $p$ (and it is
independent of the algebraically closed field).

\begin{thm}   \label{char0}
 Let $k$ be an algebraically closed field of characteristic~$0$.
 Let $X$ be the subvariety of $C_1 \times C_2 \times C_3$ consisting of those
 triples with product $1$. Then $X$ is a regular $G$-orbit and if
 $(x,y,z) \in X$, then $\langle x, y \rangle$ is a Zariski dense subgroup of
 $G(k)$.
\end{thm}

\begin{proof}
As noted, $G(k)$ has an orbit on $X$ with trivial point stabilizers. Note that
the closure of $\langle x, y \rangle$ is positive dimensional (since it
contains unipotent elements).

We claim that $H:=\langle x, y \rangle$ acts irreducibly on $V:=\Lie(G)$.
Note that $H \le G(R)$ where $R$ is some finitely generated subring of
$k$.  There is some maximal ideal $M$ of $R$ such that $x,y$ and $z$
are still in the corresponding conjugacy classes in $G(R/M)$.
By the Nullstellensatz,  $R/M$ is a finite field which we can take be of
large characteristic. By the results for finite characteristic, we know that
the image of $H$ in $G(R/M)$ is $G(\FF_p)$ and in particular acts irreducibly
on the adjoint module in characteristic $p$, whence it acts irreducibly
on $V$.  Thus, the Zariski closure of $H$ is $G$ (otherwise,
$\Lie\bar{H}$ would be a proper invariant $H$-submodule).
\end{proof}

Now fix $z \in C_3$ in $G(\QQ)$  (for example take $z = \prod u_i(1)$ where
the product is over a set of elements from root subgroups for the simple roots).
Let $D=C_G(z)$.  Note that $D$ is a connected abelian unipotent
group of dimension $r$, the rank of $G$.

Let $Y$ be the subvariety of $X$ with the third coordinate equal to $z$.
Note that $Y$ is a regular $D$-orbit (because $X$ is a regular
$G$-orbit).  Thus, $Y$ defines a $D$-torsor.  Since connected unipotent
groups have no nontrivial torsors (basically by a version of Hilbert's
Theorem 90), it follows that $Y(\QQ)$ is nonempty.   Thus, if $(x,y,z) \in X$
we see that $\langle x, y \rangle$ is conjugate to a subgroup of $G(\QQ)$.

In particular, it follows that $X(L)$ is nonempty for any field $L$ of
characteristic $0$, whence it follows trivially that $X(L)$ is a regular
$G(L)$-orbit.

\subsection{The $p$-adic case}

Here we give elementary proofs for some more general results for the
$p$-adic case.

Fix a prime $p$ and set $\ZZ_p$ to be the ring of $p$-adic integers with
field of fractions $\QQ_p$.  Let $K$ be a finite unramified extension of
$\QQ_p$ with $R$ the integral closure of $\ZZ_p$ in $K$.
Let $G$ be a split simply connected simple Chevalley group over $R$.
Let $P$ be the maximal ideal of $R$ over $p$.  Say $R/P \cong \FF_q$.
For convenience, we assume that $q > 4$.

Let $N_j$ be the congruence kernel of the natural map from
$G(R)$ to $G(R/P^j)$ and set $N=N_1$.

\begin{lem}   \label{lem:lifting}
 Let $x_1, \ldots, x_r \in G(R)$ with $\prod x_i \in N$ and
 set $y_i = x_i \mod N$.  Assume that $\langle y_1, \ldots, y_r \rangle = G/N$.
 Then  there are conjugates $w_i$ of $x_i$ such that
 $\prod w_i = 1$ and $x_iN=w_iN$.  Moreover,
 $\langle x_1,\ldots, x_r \rangle$ and  $\langle w_1, \ldots,  w_r \rangle$
 are dense in $G(R)$ in the $p$-adic topology.
\end{lem}

\begin{proof}
By induction and a straightforward compactness argument,
it suffices to assume that $\prod x_i \in N_j$ and then show
that we can choose $n_{ij} \in N_j$ so that
$\prod x_i^{n_{ij}} \in N_{j+1}$.   This follows from the fact that
$\langle y_1, \ldots, y_r \rangle = G/N$ and $G/N$ has no
covariants on $N_j/N_{j+1} \cong \Lie(G/N)$ \cite[3.5]{weigel}.

The fact that $ \langle w_1, \ldots,  w_r \rangle$
is dense in $G(R)$ follows from the fact that
$N$ is contained in the Frattini subgroup of $G(R)$ \cite{weigel}.
\end{proof}

\begin{rem}
If  $y_1,\ldots,y_r\in G(R)/N$ with $\prod y_i = 1$,
the order of $y_i$ prime to $p$
and $\langle y_1, \ldots, y_r \rangle = G/N$, then we can
lift each $y_i$ to an element $x_i \in G(R)$ with $y_i = x_iN$
and so the previous result applies in this case. See
\cite{GT} for a more general result.
\end{rem}

\subsection{Completion of the proof of Theorem \ref{thm:char 0}}
Now return to the set-up in subsection \ref{ss:fields}.
Let $K = \QQ_p$ with $p$ a good prime for $G$.
Let $D_i$, $i = 1, 2, 3$ be the corresponding conjugacy classes in
$G(\FF_p) = G(R)/N$ and let $C_i$ be the classes in $G(\QQ)$.
By Lemma~\ref{lem:lifting}, we can choose $w_i \in C_i \cap G(R)$ with
$w_1w_2w_3=1$. Note that if $w \in Y(R)$, then since
$\Gamma(w):=\langle w\rangle$ is dense
in $G(R)$,  it follows that $G(R)$ acts regularly on those elements in
$X(R)$ which generate a dense subgroup of  $G(R)$
(since $G(R)$ is self normalizing in $G(\QQ_p)$ --- we will not use
this fact in what follows).

We next want to consider integrality questions.
Let $S=\ZZ[1/m]$ where $m$ is the product of the bad primes of $G$.
By Theorem \ref{char0}, we may choose $x \in X(\QQ)$.
Thus, $x \in X(\ZZ[1/N])$ for some positive (squarefree)
integer $N$.  Suppose that some good prime $p$ divides $N$.
By Lemma~\ref{lem:lifting},
we may choose  $y \in X(\ZZ_p)$.   So  $y = g.x$ for some $g \in G(\QQ_p)$.

Note that $G(\QQ_p) = G(\ZZ[1/p])G(\ZZ_p)$ (this is because
$G(\ZZ_p)$ is open in $G(\QQ_p)$ in the $p$-adic topology
and $G(\ZZ[1/p])$ is dense in $G(\QQ_p)$ (since $\ZZ[1/p]$ is dense in
$\QQ_p$ and $G(\QQ_p)$ is generated by root subgroups each isomorphic to
$\QQ_p)$). So write $g = g_1 g_2$ where $g_1 \in G(\ZZ[1/p])$ and
$g_2 \in G(\ZZ_p)$. Thus, $g_1^{-1}.y = g_2.x$, and so
$w:=g_1^{-1}.y = g_2.x \in G(\ZZ[1/N])\cap G(\ZZ_p)=G(\ZZ[1/N'])$
where $N=pN'$. Moreover, we see that $\Gamma(w)$ surjects
onto  $G(\FF_r)$ for any $r$ not dividing $N'$ (because $y$ and so
$g_1^{-1}.y$ have this property and also for $r=p$ since $g_2.x$ has this
property).

Continuing in this manner, we see that we can produce such an embedding
into $G(S)$ as required.  Thus, we have proved Theorem \ref{thm:char 0}.
\medskip

If $G=G_2$, Dettweiler and Reiter \cite{DR10} exhibited a triple in $X(\ZZ)$.
If $G=E_8$, we do not know if the group is in fact conjugate to a
subgroup of $G(\ZZ)$.

Suppose that $x = (x_1, x_2, x_3) \in X(\ZZ)$.  Let $\Gamma= \Gamma(x)$.
Let  $W=\Lie(G(\FF_2))$ and $V= \Lie(G(\CC))$.
It is clear that $\dim [x_i,W] \le \dim [x_i,V]$ and since $x_1$ is an
involution, $\dim [x_1, W] \le (1/2) \dim W < \dim [x_1,V]$.  Thus
$$\sum \dim [x_i,W] < \sum \dim [x_i,V] = 2 \dim W.$$
By  Scott's Lemma \cite{scott} it follows that the image of $\Gamma$ in
$G(\FF_2)$ either has fixed points or covariants on $W$.  Since $G(\FF_2)$
has no fixed points on $W$, it follows that the image of $\Gamma$ is a proper
subgroup of $G(\FF_2)$. Indeed, the same shows that the image of $\Gamma$ is
contained in a proper positive dimensional subgroup of $G(\overline{\FF_2})$.

\section{Remarks on $F_4$}   \label{sec:F4}

Let $k$ be an algebraically closed field of characteristic $p>3$ and
$G=F_4(k)$. Let $C_1$ be the conjugacy class of $G$ consisting of involutions
with centralizer $A_1(k)C_3(k)$, $C_2$ the conjugacy class of unipotent
elements $A_1+\tilde{A_1}$ and $C_3$ the conjugacy class of regular unipotent
elements. We set
$$X:=\{(x,y,z)\in C_1 \times C_2 \times C_3\mid xyz=1\}.$$

The character theory proof goes through for this set of triples showing
that $\dim X = 52$ and there is at most one component of dimension~$52$.
By standard intersection theory, any component of $X$ has dimension at
least $52$, whence:

\begin{prop}   \label{lem:dimf41}
 $X$ is an irreducible variety of dimension $52$.
\end{prop}

If $x=(x_1, x_2, x_3) \in X$, let $\Gamma(x) = \langle x_1, x_2 \rangle$. 

Unfortunately, no triple in $X$ generates an $F_4(p)$ because elements of $C_1$ have
a $14$-dimensional fixed space on the $26$-dimensional module $V$ for
$G$ and elements of $C_2$ have a $14$ dimensional fixed space on $V$
\cite{Law95}. Thus if $x_i \in C_i$, $\langle x_1, x_2 \rangle$ has at least
a $2$-dimensional fixed space on $V$.  Since $x_3$ has a $2$-dimensional
fixed space on $V$, this is precisely the fixed space of $\Gamma(x)$.
Choose a $B_3$-parabolic subgroup $P$ containing $x_3$.  Then
$P$ has a unique $1$-dimensional invariant space on $V$, whence
it follows that $\Gamma(x) < P$.  

We can show:

\begin{thm} \label{thm:f4triple}
 If $(x_1, x_2, x_3) \in X$, then $\langle x_1, x_2, x_3 \rangle = RG_2(p)$
 where $R$ is nilpotent of class $2$ and has order $p^{14}$.   Moreover,
 $X$ is a single regular $G$-orbit. 
\end{thm}

\begin{proof} 
Consider $G_2(k) < B_3(k) < QB_3(k) < P < G$ 
where $P$ is a maximal parabolic subgroup and 
$Q$ is the unipotent radical of $P$. 

By \cite[Prop.~4.5]{CKS}, $[Q,Q]$ is the natural $7$-dimensional 
module for $B_3(k)$ and $A :=Q/ [Q,Q]$ is the $8$-dimensional 
spin module for $B_3(k)$. Since $G_2(k)$ has only nontrivial 
irreducible modules of dimension $7$ or dimension at least $14$, 
it follows that  as $G_2(k)$-modules, $[Q,Q]$ is irreducible and
$A \cong k \oplus B$ with $B$ a $7$-dimensional irreducible module for $G_2(k)$ 
(it must split because $A$ is self dual). Note that
an element of $Q$ fixed by $G_2(k)$ (even modulo $[Q,Q]$) 
is not of the form $u_4(t)$ for some $t \ne 0$ (because 
the stabilizer of such an element is a maximal parabolic 
subgroup of $B_3(k)$ and so does not contain $G_2(k)$). 

Let $x \in G_2(k)$ be an involution, $y \in G_2(k)$ a unipotent element
in the class $\tilde{A} _1$ and $z$ in the class of regular unipotent elements
of $G_2(k)$ with $xyz=1$. Then,  by the rigidity result for $G_2(k)$,
$\langle x, y \rangle \cong G_2(p)$ and so we may assume that
$\langle x, y \rangle = G_2(p)$.  Moreover, by conjugating in $G_2(p)$,
we may assume that $z = u_1(1)u_2(1)u_3(1)$ where 
$u_i( t ), 1 \le i \le 3$, are the root subgroups corresponding 
to the simple roots of $G$ inside $B_3(k)$. 

It is straightforward to see that $\dim C_G(x)=24$ and that $y$ has
the same Jordan block structure on the adjoint module for $G$
as do elements in $C_2$.    This implies that $x \in C_1$
and $y \in C_2$. 

By the remarks above, $[G_2(p), Q]$ is a codimension $1$ subgroup 
of $Q$.  Indeed, setting $R=Q(p)$, we see that $R_0 :=[G_2(p), R]$ 
has order $p^{14}$ and has index $p$ in $Q(p)$. It follows 
easily that every element of $R_0$ can be written as $[x,q_1][y,q_2]$ 
for some $q_1, q_2 \in Q(p)$. In particular, $u_4(-1) = [x,q_1] [y,q_2]v$ 
where $q_1, q_2 \in Q(p)$ and 
$v$ is a product of root elements in $R$ corresponding to nonsimple 
roots. Let $q_3 \in R$ with $x^{q_1}y^{q_2}(q_3z) =1$. 
Then $q_3z= \prod_{i=1}^4 u_i(1) v'$ where $v'$ is a 
product of root elements in $R$ corresponding to nonsimple roots. 
In particular $q_3z$ is a regular unipotent element of $G$. 

Thus, we have produced a triple $y:=(y_1, y_2, y_3) 
= (x^{q_1}, y^{q_2}, q_3z) \in X$. Note that $H := \Gamma(y)
\le [R,G_2(p)]G_2(p)$. Since $H$ contains a regular unipotent element, 
$H[R,R]/[R,R]$ intersects $R/[R,R]$ nontrivially. The argument 
above shows that $x$ does not act trivially on this intersection, 
whence $H[R,R]/[R,R]$ contains the hyperplane $R_0[R,R]/[R,R]$ of $R/[R,R]$. 
Note that $H \cap [R,R] \ne 1$ for otherwise $H \cap [R,R]$ is abelian 
of order at least $p^7$ and centralizes $[R,R] < Z(Q)$. Then 
$ (H \cap R) [R,R]$ is abelian of order $p^{14}$ (and this is 
not possible, either by inspection or by  \cite[Table 3.3.1]{GLS3}). 
Since $H$ acts irreducibly on $[R,R]$, this implies that $[R,R] \le H$. 
Thus, $H \cap R = [G_2(p),R]$ has index $p$ in $R$. 

We next claim that $C:=C_G(H)=1$.   Suppose not.
 Since $G_2(p)$ is self centralizing 
in $B_3(k)$ and $C \le C_G(y_3) < P$, it follows that $C \le Q$.
Since $G_2(p)$ acts without fixed points on $[G_2(p),Q]$, it follows
that  $C \cap [G_2(p),Q]=1$.   Let $T$ be the torus centralizing $B_3(k)$.
Then $T$ normalizes $H$ (because it centralizes $G_2(k)$ and normalizes
$Q$).   Thus,  $T$ also normalizes $C$.  Since $C_Q(T)=1$, it follows that
$C$ has positive dimension and that $Q=[G_2(p),Q]C$, whence
$C$ centralizes $Q$.  Thus, $C \le Z(Q)=[Q,Q]$, a contradiction. 

We next show that any $x =(x_1, x_2, x_3) 
\in X$ is as above.   As noted, 
we may assume that $x_3 \in P$ and so $H \le P_3$.

Arguing as in the $E_8$ case, we see that $HQ/Q$ cannot be contained 
in a parabolic subgroup of $B_3(k)$. It follows easily from the fact
that $HQ/Q$ contains a regular unipotent element of $B_3(k)$  that 
either $HQ/O$ contains a conjugate of $B_3(p)$ or is contained 
in $G_2(k)$. Arguing as above, we see that in 
$HQ/Q$,  the $x_iQ$ are precisely in the rigid classes for $G_2$ 
(inside $B_3$ ). Note that on the $8$-dimensional module $W$ for 
$B_3$, $\sum \dim [x_i,W] < 16$, whence by Scott's 
Lemma, $H$ does not act irreducibly on $W$ and so 
$HQ/Q  \le G_2(k)$.  By the rigidity result for $G_2(k)$, 
this implies that $HQ/Q  \cong G_2(p)$. Now we argue as above 
to conclude that $H=[R,G_2(p)]G_2(p)$ and has trivial centralizer 
in $G$. 

By Proposition~\ref{lem:dimf41} we conclude that since $X$ is an irreducible
variety of dimension $52$ and every orbit of $G$ on $X$ has dimension $52$,
$X$ is a single $G$ -orbit. 

\end{proof}

There are several other candidates for rigid triples (satisfying the
necessary condition that $\sum\dim C_i = 2 \dim G$) but in all cases there
seem to be technical difficulties in establishing either the character results
we need or the generation result or both.  In all cases, $C_3$ will be the
regular unipotent class. The possibilities are:

\begin{enumerate}
\item  $C_1$ consists of involutions of type $B_4$ and $C_2$
  consists of unipotent elements of type $F_4(A_3)$.
\item  $C_1$ consists of involutions of type $A_1C_3$ and
  $C_2$ is the class of elements which are a commuting product
  of a $B_4$-involution and a long root element.
\item $C_1$ consists of unipotent elements of type $A_1+\tilde A_1$ and
  $C_2$ is the class of elements which are a commuting
  product of a $B_4$-involution and a long root element.
\end{enumerate}

The second  triple was suggested by Yun.


\begin{thebibliography}{GPPS99}

\bibitem[Asch]{asch}
{\sc M. Aschbacher}, The maximal subgroups of $E_6$. Preprint.

\bibitem[Ca93]{Ca}
{\sc R. W. Carter}, \emph{Finite Groups of Lie Type. Conjugacy Classes and
  Complex Characters}. Wiley Classics Library. John Wiley \& Sons,
  Chichester, 1993.
  
  

\bibitem[CLSS92]{CLSS}
{\sc A. M. Cohen, M. W. Liebeck, J. Saxl, G. M. Seitz}, The local maximal
  subgroups of exceptional groups of Lie type, finite and algebraic.
  Proc. London Math. Soc. (3) {\bf 64} (1992), 21--48.

\bibitem[Co81]{Coop}
{\sc B. N. Cooperstein}, Maximal subgroups of $G_2(2^n)$. J. Algebra {\bf 70}
  (1981), 23--36.
  
\bibitem[CKS76]{CKS} 
{\sc C. Curtis, W. M. Kantor, G. M. Seitz}, The 2-transitive permutation
  representations of the finite Chevalley groups. Trans. Amer. Math. Soc.
  {\bf 218} (1976), 1--59. 

\bibitem[DR10]{DR10}
{\sc M. Dettweiler, S. Reiter}, Rigid local systems and motives of type $G_2$.
  With an appendix by Michael Dettweiler and Nicholas M. Katz.
  Compos. Math. {\bf 146} (2010), 929--963.

\bibitem[DM91]{DM91}
{\sc F. Digne, J. Michel}, \emph{Representations of Finite Groups of Lie Type}.
 LMS Student Texts, 21. Cambridge University Press, Cambridge, 1991.

\bibitem[Di12]{Di}
{\sc J. DiMuro}, On prime power order elements of general linear groups.
  J. Algebra, to appear.

\bibitem[FF85]{FF}
{\sc W. Feit, P. Fong}, Rational rigidity of $G_2(p)$ for any prime $p>5$.
  Pp.~323--326 in: \emph{Proceedings of the Rutgers Group Theory Year,
  1983--1984} (New Brunswick, N.J., 1983--1984). Cambridge Univ. Press,
  Cambridge, 1985.
  
\bibitem[GLS98]{GLS3}
{\sc D. Gorenstein, R. Lyons, R. Solomon}, \emph{The Classification of the
  Finite Simple Groups. Number~3.} Mathematical Surveys and Monographs,
  American Mathematical Society, Providence, RI, 1998.

\bibitem[GPPS99]{GPPS}
{\sc R. M. Guralnick, T. Penttila, C. Praeger, J. Saxl},
  Linear groups with orders having certain large prime divisors.
  Proc. London Math. Soc. (3) {\bf 78} (1999), 167--214.

\bibitem[GT12]{GT}
{\sc R. M. Guralnick, P. Tiep}, Lifting in Frattini covers and a
  characterization of finite solvable groups. Preprint, arXiv:1112.4559.

\bibitem[Kl88]{PBK}
{\sc P. B. Kleidman}, The maximal subgroups of the Chevalley groups $G_2(q)$
  with $q$ odd, the Ree groups ${^2}G_2(q)$, and their automorphism groups.
  J. Algebra {\bf 117} (1988), 30--71.

\bibitem[La95]{Law95}
{\sc R. Lawther}, Jordan block sizes of unipotent elements in exceptional
  algebraic groups. Comm. Algebra {\bf 23} (1995), 4125--4156.

\bibitem[La09]{Law09}
{\sc R. Lawther}, Unipotent classes in maximal subgroups of exceptional
  algebraic groups. J. Algebra {\bf 322} (2009), 270--293.

\bibitem[La12]{Law12}
{\sc R. Lawther}, Sublattices generated by root differences. Preprint.

\bibitem[LS99]{LS99}
{\sc M. W. Liebeck, G. M. Seitz}, On finite subgroups of exceptional
  algebraic groups. J. reine angew. Math. {\bf 515} (1999), 25--72.

\bibitem[LS03]{LS03}
{\sc M. W. Liebeck,  G. M. Seitz}, A survey of maximal subgroups of
  exceptional groups of Lie type. Pp.~139--146 in: \emph{Groups, Combinatorics
  \& Geometry} (Durham, 2001). World Sci. Publ., River Edge, NJ, 2003.

\bibitem[Lu86]{LuV}
{\sc G. Lusztig}, Character sheaves. V. Adv. in Math. {\bf61} (1986), 103--155.

\bibitem[Lu88]{Lu88}
{\sc G. Lusztig}, On the representations of reductive groups with disconnected
  centre. Ast\'erisque {\bf168} (1988), 157--166.

\bibitem[Ma90]{Ma}
{\sc K. Magaard}, \emph{The Maximal Subgroups of $F_4(F)$ Where $F$ is a
  Finite or Algebraically Closed Field of Characteristic $\ne 2, 3$}.
  Ph. D. thesis, Caltech, 1990.

\bibitem[MM99]{MM}
{\sc G. Malle, B. H. Matzat}, \emph{Inverse Galois Theory.} Springer
  Monographs in Mathematics. Springer-Verlag, Berlin, 1999.

\bibitem[MT11]{MT}
{\sc G. Malle, D. Testerman}, \emph{Linear Algebraic Groups and Finite
  Groups of Lie Type}. Cambridge Studies in Advanced Mathematics, 133.
  Cambridge University Press, Cambridge, 2011.

\bibitem[Mi]{MChev}
{\sc J. Michel}, The GAP-part of the Chevie system. GAP 3-package available for
  download from  http://people.math.jussieu.fr/\~{}jmichel/chevie/chevie.html

\bibitem[SS97]{SS97}
{\sc J. Saxl, G. M. Seitz}, Subgroups of algebraic groups containing regular
  unipotent elements. J. London Math. Soc. {\bf55} (1997), 370--386.

\bibitem[Sc77]{scott}
{\sc L. L. Scott}, Matrices and cohomology. Ann. of Math. {\bf105} (1977),
  473--492.

\bibitem[ST90]{ST1}
{\sc G. M. Seitz, D. Testerman}, Extending morphisms from finite to algebraic
  groups. J. Algebra {\bf 131} (1990), 559--574.

\bibitem[ST93]{ST2}
{\sc G. M. Seitz, D. Testerman}, Subgroups of type $A_1$ containing semiregular
  unipotent elements. J. Algebra {\bf 196} (1997), 595--619.

\bibitem[Si93]{sin}
{\sc P. Sin}, Extensions of simple modules for $G_2(3^n)$ and ${^2}G_2(3^m)$.
  Proc. London Math. Soc. {\bf 66} (1993), 327--357.

\bibitem[Sp85]{Spa}
{\sc N. Spaltenstein}, On the generalized Springer correspondence for
  exceptional groups. Pp.~317--338 in: \emph{Algebraic Groups and Related
  Topics} (Kyoto/Nagoya, 1983), Adv. Stud. Pure Math., 6. North-Holland,
  Amsterdam, 1985.

\bibitem[Su09]{Su}
{\sc I. Suprunenko}, \emph{The Minimal Polynomials of Unipotent Elements in
  Irreducible Representations of the Classical Groups in Odd Characteristic.}
  Mem. Amer. Math. Soc. {\bf 200} (2009), no. 939. Amer. Math. Soc., Providence,
  RI.

\bibitem[Th85]{Th}
{\sc J. G. Thompson}, Rational rigidity of $G_2(5)$. Pp.~321--322 in:
  \emph{Proceedings of the Rutgers Group Theory Year, 1983--1984}
  (New Brunswick, N.J., 1983--1984). Cambridge Univ. Press, Cambridge, 1985.

\bibitem[We96]{weigel}
{\sc T. Weigel}, On the profinite completion of arithmetic groups of split
  type. Pp.~79--101 in: \emph{Lois d'Alg\`ebres et Vari\'et\'es Alg\'ebriques}
  (Colmar, 1991), Travaux en Cours, 50. Hermann, Paris, 1996.

\bibitem[Yu12]{Yu}
{\sc Z. Yun}, Motives with exceptional Galois groups and the inverse Galois
  problem. Preprint, arXiv:1112.2434v1.
\end{thebibliography}
\end{document}